\documentclass[12pt,a4paper,reqno]{amsart}
\usepackage{amsmath}
\usepackage{amsfonts}
\usepackage{amssymb}
\usepackage{amsthm}
\usepackage{enumerate}
\usepackage{centernot}
\usepackage{fullpage}
\usepackage{mathrsfs}
\usepackage{graphicx}
\usepackage{color}

\newcommand{\upperRomannumeral}[1]{\uppercase\expandafter{\romannumeral#1}}

\theoremstyle{plain}

  \newtheorem{proposition}[]{Proposition}
  \newtheorem{lemma}[]{Lemma}
  \newtheorem{theorem}[]{Theorem}
  \newtheorem{corollary}[]{Corollary}
  \newtheorem{conjecture}[]{Conjecture}
  \newtheorem{remark}[]{Remark}

\title[Curie-Weiss perturbation]{Ising model with Curie-Weiss perturbation}

\author{Federico Camia}
\address{Division of Science, NYU Abu Dhabi, Saadiyat Island, Abu Dhabi, UAE \& Courant Institute of Mathematical Sciences, New York University, 251 Mercer st, New York, NY 10012, USA \& Department of Mathematics, Faculty of Science, Vrije Universiteit Amsterdam, De Boelelaan 1111, 1081 HV Amsterdam, The Netherlands.}
\email{federico.camia@nyu.edu}
\author{Jianping Jiang}
\address{Yanqi Lake Beijing Institute of Mathematical Sciences and Applications, Building 11, Yanqi Island, Yanqi Lake West Road, Beijing 101408, China \& Yau Mathematical Sciences Center, Tsinghua University, Beijing 100084, China.}
\email{jianpingjiang11@gmail.com}
\author{Charles M. Newman}
\address{Courant Institute of Mathematical Sciences, New York University, 251 Mercer st, New York, NY 10012, USA \& NYU-ECNU Institute of Mathematical ciences at NYU Shanghai, 3663 Zhongshan Road North, Shanghai 200062, China.}
\email{newman@cims.nyu.edu}

\begin{document}
\begin{abstract}
Consider the nearest-neighbor Ising model on $\Lambda_n:=[-n,n]^d\cap\mathbb{Z}^d$ at inverse temperature $\beta\geq 0$ with free boundary conditions, and let $Y_n(\sigma):=\sum_{u\in\Lambda_n}\sigma_u$ be its total magnetization. Let $X_n$ be the total magnetization perturbed by a critical Curie-Weiss interaction, i.e.,
\begin{equation*}
\frac{d F_{X_n}}{d F_{Y_n}}(x):=\frac{\exp[x^2/\left(2\langle Y_n^2 \rangle_{\Lambda_n,\beta}\right)]}{\left\langle\exp[Y_n^2/\left(2\langle Y_n^2\rangle_{\Lambda_n,\beta}\right)]\right\rangle_{\Lambda_n,\beta}},
\end{equation*}
where $F_{X_n}$ and $F_{Y_n}$ are the distribution functions for $X_n$ and $Y_n$ respectively. We prove that for any $d\geq 4$ and $\beta\in[0,\beta_c(d)]$ where $\beta_c(d)$ is the critical inverse temperature, any subsequential limit (in distribution) of $\{X_n/\sqrt{\mathbb{E}\left(X_n^2\right)}:n\in\mathbb{N}\}$ has an analytic density (say, $f_X$) all of whose zeros are pure imaginary, and $f_X$ has an explicit expression in terms of the asymptotic behavior of zeros for the moment generating function of $Y_n$. We also prove that for any $d\geq 1$ and then for $\beta$ small,
\begin{equation*}
	f_X(x)=K\exp(-C^4x^4),
\end{equation*}
where $C=\sqrt{\Gamma(3/4)/\Gamma(1/4)}$ and $K=\sqrt{\Gamma(3/4)}/(4\Gamma(5/4)^{3/2})$. Possible connections between $f_X$ and the high-dimensional critical Ising model with periodic boundary conditions are discussed.
\end{abstract}

\maketitle

\section{Introduction}
\subsection{Overview and motivation}
It has been known since \cite{SG73}  that the renormalized total magnetization in the critical Curie-Weiss model converges in distribution to a random variable with density $C_1\exp(-C_2x^4)$. See \cite{EN78,ENR80} for various extensions of this classical result as the underlying single spin distribution at $\beta=0$ is varied. Part of the current paper (when $d\geq 1$ and $\beta$ is small) can be viewed as studying extensions of this classical result by including some nearest-neighbor interactions.

The effect of periodic and free boundary conditions for the high-dimensional Ising model in finite domains has been a  long-standing debate in the physics literature; see, e.g., \cite{WY14,GEZGD17,ZGDG20} for some recent results and references therein for many more studies. For the critical Ising model on $\Lambda_n:=[-n,n]^d\cap\mathbb{Z}^d$ where $d>4$ with periodic boundary conditions, the two-point function is conjectured in \cite{Pap06} to have a ``plateau":  it behaves like the infinite-volume counterpart for small distance before leveling off at an order of $n^{-d/2}$ for large distance. This plateau lower bound was proved in Theorem~1 of \cite{Pap06} and the upper bound remains open; the conjecture was numerically verified in \cite{GEZGD17}. Similar plateaus were proved for simple random walk for $d>2$ in \cite{ZGFDG18,Sla20,ZGDG20}, for weakly self-avoiding walk for $d>4$ in \cite{Sla20}, and for nearest-neighbor percolation for $d\geq 11$ in \cite{HMS21}. On the contrary, it was proved in \cite{CJN21} that such a plateau does not exist in the high-dimensional critical Ising model with free boundary conditions. 

Loosely speaking, the main conclusion in \cite{Pap06} was that a macroscopic (or bulk) quantity in the high-dimensional Ising model with periodic boundary conditions ``parallels more closely the complete graph paradigm". For example, it was proved in \cite{Pap06} that the critical susceptibility on $\Lambda_n$ in dimensions $d>4$ with periodic boundary conditions has a lower bound $n^{d/2}$ (which is the order of the critical susceptibility in the Curie-Weiss model); while it behaves like $n^2$ for free boundary conditions (see \cite{CJN21} for a proof).  

A main motivation of the current paper is to study the scaling limit of the critical Ising magnetization field in dimension $d>4$ with periodic boundary conditions. The results in \cite{Pap06} suggest that this scaling limit is non-Gaussian, and on page 37 of \cite{Pap06}, the author mentions that the macroscopic behavior in the high dimensional Ising model with periodic boundary conditions is expected to behave similarly to the Curie-Weiss model. So in this paper, we study the Ising model perturbed by a critical Curie-Weiss interaction. We prove that for any $d\geq 1$ and $\beta$ small, the limit of the total magnetization in this perturbed model is the same as in the standard critical Curie-Weiss model, and we conjecture that this should be true for any $\beta\in[0,\beta_c(d))$ where $\beta_c(d)$ is the critical inverse temperature of the classical Ising model on $\mathbb{Z}^d$. For $d\geq 4$ and $\beta=\beta_c(d)$, we prove the partial result that any subsequential limit of this perturbed total magnetization is non-Gaussian and has an analytic density all of whose zeros are pure imaginary.  We think that there may be a close connection between periodic boundary conditions and Curie-Weiss perturbation scaling limits in the high-dimensional Ising model (see Conjecture \ref{conj1} and \eqref{eq:cwpp} of Section~\ref{subsec:HighDIsing} below) and thus the analysis of the Curie-Weiss case may help understand the periodic boundary conditions case.

\subsection{Main results}
Consider the Ising model on $\Lambda_n$ at inverse temperature $\beta\geq 0$ with free boundary conditions; let $\langle\cdot\rangle_{\Lambda_n,\beta}$ denote its expectation. See \eqref{eq:expectationdef}-\eqref{eq:partf} (with $h=0$)  in the Appendix for explicit definitions. Let 
\begin{equation}
	Y_n(\sigma):=\sum_{u\in\Lambda_n}\sigma_u
\end{equation}
be the total magnetization. For any $d\geq 1$ and $\beta\in[0,\beta_c(d))$, it is not hard to prove (see Lemma \ref{lem:X_nY_n} below) that
\begin{equation}\label{eq:Y_ntonormal}
	\frac{Y_n}{\sqrt{\langle Y_n^2\rangle_{\Lambda_n,\beta}}}\Longrightarrow \text{ a standard normal distribution as }n\rightarrow\infty, 
\end{equation}
where $\Longrightarrow$ denotes convergence in distribution. By the main results in \cite{Aiz82,Fro82,ADC21}, the last weak convergence also holds for any $d\geq 4$ and $\beta=\beta_c(d)$.  For $d=2$ and $\beta=\beta_c(2)$, it was proved in \cite{CGN16} that $Y_n/\sqrt{\langle Y_n^2\rangle_{\Lambda_n,\beta}}$ converges in distribution to $Y$ with
\begin{equation}
	\lim_{x\rightarrow\infty}\frac{\ln\mathbb{P}(Y>x)}{x^{16}}=-c,
\end{equation}
where $c\in(0,\infty)$. For $d=3$ and $\beta=\beta_c(3)$, the limit behavior of $Y_n/\sqrt{\langle Y_n^2\rangle_{\Lambda_n,\beta}}$ is unknown. For any $d\geq 2$ and $\beta>\beta_c(d)$, it is believed that
\begin{equation}\label{eq:Y_nsuper}
	\frac{Y_n}{\sqrt{\langle Y_n^2\rangle_{\Lambda_n,\beta}}}\Longrightarrow \frac{1}{2}\delta_{-1}+\frac{1}{2}\delta_{+1},
\end{equation}
where $\delta_{\pm 1}$ is the unit Dirac point measure at $\pm 1$. A well-known result due to Lee and Yang \cite{LY52} is that all zeros of the moment generating function of $Y_n$, $\langle\exp(zY_n)\rangle_{\Lambda_n,\beta}$ or equivalently of the partition function (see \eqref{eq:partf}), are pure imaginary; so we may assume that all the zeros are $\{\pm i\alpha_{j,n}:j\geq 1\}$ (listed with multiplicities, if any) such that
\begin{equation}
	0<\alpha_{1,n}\leq\alpha_{2,n}\leq\dots.
\end{equation}

We consider a random variable $X_n$ whose distribution is obtained from that of $Y_n$ by adding a Curie-Weiss interaction:
\begin{equation}
	\frac{d F_{X_n}}{d F_{Y_n}}(x)=\frac{\exp[x^2/(2\langle Y_n^2\rangle_{\Lambda_n,\beta})]}{\left\langle \exp[Y_n^2/(2\langle Y_n^2\rangle_{\Lambda_n,\beta})]\right\rangle_{\Lambda_n,\beta}},
\end{equation}
where $F_{X_n}$ and $F_{Y_n}$ are the distribution functions for $X_n$ and $Y_n$ respectively. Equivalently, $X_n$ is the total magnetization for the perturbed Ising model with Hamiltonian
\begin{equation}\label{eq:Hpert}
	H_{\Lambda_n,\beta}(\sigma):=-\beta\sum_{\{u,v\}}\sigma_u\sigma_v-\frac{Y^2_n(\sigma)}{2\langle Y_n^2\rangle_{\Lambda_n,\beta}},\qquad \sigma\in\{-1,+1\}^{\Lambda_n}.
\end{equation}
When $\beta=0$, the form of \eqref{eq:Hpert} means that this perturbed model is exactly the critical Curie-Weiss model. Let $\mathbb{E}_{\Lambda_n,\beta}$ denote the expectation for this perturbed Ising model. The \textbf{cumulants}, $u_n(X)$, of a random variable $X$ are defined (if they exist) by
\begin{equation}
	u_n(X):=\left.\frac{d^n\ln\mathbb{E} \exp(zX)}{d z^n}\right\vert_{z=0}.
\end{equation}

Here is our main result.
\begin{theorem}\label{thm1}
	For $d\geq 1$ and $\beta\in[0,\beta_c(d))$, and also for $d\geq 4$ and $\beta=\beta_c(d)$, any subsequential limit in distribution of $\{X_n/\sqrt{\mathbb{E}_{\Lambda_n,\beta}(X_n^2)}:n\in\mathbb{N}\}$ has an analytic density $f_X$ all of whose zeros in $\mathbb{C}$ are pure imaginary. More precisely, suppose that there exists a subsequence $\{n_k:k\in\mathbb{N}\}$ and a random variable $X$ such that
    \begin{equation}
       X_{n_k}/\sqrt{\mathbb{E}_{\Lambda_{n_k},\beta}(X_{n_k}^2)}\Longrightarrow X \text{ as } k\rightarrow\infty. 
     \end{equation}
    Then $X$ has a density function given by
	\begin{equation}\label{eq:f_X}
		f_X(x)=C_3f_W(C_3x),
	\end{equation}
	where $C_3:=\sqrt{\mathbb{E}(W^2)}\in(0,\infty)$ and
	\begin{equation}\label{eq:f_W}
		f_W(x)=K\exp[-\kappa_1x^4]\prod_{j\geq 1}\left[\left(1+\frac{x^2}{a_j^2}\right)\exp\left(-\frac{x^2}{a_j^2}\right)\right],~x\in\mathbb{R},
	\end{equation}
	where  
	\begin{equation}
		a_j:=\lim_{k\rightarrow\infty}\alpha_{j,n_k}\left[-u_4(Y_{n_k})/4!\right]^{1/4} \text{ for each }j\geq 1,
	\end{equation}
	\begin{equation}\label{eq:kappa_1def}
		\kappa_1:=1-\frac{\sum_{j\geq 1}a_j^{-4}}{2},
	\end{equation}
	and $K\in(0,\infty)$ is such that $f_W$ is a probability density function; $\{a_j:j\geq 1\}$ may be empty, finite or infinite.
\end{theorem}

\begin{remark}
	Theorem \ref{thm1} treats the critical Curie-Weiss perturbation of the Ising model with $\beta\leq\beta_c(d)$. If the Curie-Weiss perturbation is subcritical, then like in the $\beta=0$ case, $X$ is expected to be standard normal. Similarly, in the supercritical case, we expect $X$ to be distributed as a sum of two Dirac point measures as in \eqref{eq:Y_nsuper}.
\end{remark}

\begin{remark}
	By Fatou's lemma and \eqref{eq:cumdef} below, we have
	\begin{equation}
		\sum_{j\geq 1}a_j^{-4}=\sum_{j\geq 1}\liminf_{k\rightarrow\infty}\alpha_{j,n_k}^{-4}\left[-u_4(Y_{n_k})/4!\right]^{-1}\leq \liminf_{k\rightarrow\infty}\sum_{j\geq1}\alpha_{j,n_k}^{-4}\left[-u_4(Y_{n_k})/4!\right]^{-1}=2.
	\end{equation}
So \eqref{eq:kappa_1def} implies that $\kappa_1\in[0,1]$.
\end{remark}

In \cite{YL52}, Yang and Lee argued that when $\beta<\beta_c(d)$, the partition function of the Ising model (as a function of $\exp(-2h)$ where $h$ is the external field defined in the Appendix)  should be nonzero in a neighborhood of $\exp(i0)$, which means that $\alpha_{1,n}$ should be uniformly (in $n$) bounded away from $0$. This has been proved rigorously in \cite{Rue71} (see also Theorem~A of \cite{PR20}) when $\beta$ is small but positive. It is not hard to show that $-u_4(Y_n)$ grows at least linearly in $|\Lambda_n|$. So we have the following corollary, with a complete proof presented in Section \ref{subsec:cor}.
\begin{corollary}\label{cor1}
	For any $d\geq 1$, there exists $\beta_1(d)\in(0,\beta_c(d)]$ such that for any $\beta\in[0,\beta_1(d))$, we have that as $n\rightarrow\infty$,
	\begin{equation}
		\frac{X_n}{\sqrt{\mathbb{E}_{\Lambda_n,\beta}(X_n^2)}}\Longrightarrow \text{ a random variable with density } C\exp(-C^4x^4)/\int_{-\infty}^{\infty} \exp(-t^4)dt ,
	\end{equation}
where 
    \begin{equation}
	C:=\sqrt{\Gamma(3/4)/\Gamma(1/4)}
	\end{equation}
with $\Gamma$ denoting the gamma function.
\end{corollary}
\begin{remark}\label{rem:Y_ngen}
	\eqref{eq:f_X} and \eqref{eq:f_W} imply that there exist $C_4, C_5\in(0,\infty)$ such that
	\begin{equation}
		f_X(x)=C_4-C_5x^4+O(x^6) \text{ as }x\rightarrow0.
	\end{equation}
	(See, e.g., \eqref{eq:f_Wsmallx1} below for a derivation.)
	So any subsequential limit of $\{X_n/\sqrt{\mathbb{E}_{\Lambda_n,\beta}(X_n^2)}:n\in\mathbb{N}\}$ is not Gaussian. Theorem \ref{thm1} and Corollary \ref{cor1} also hold (with a similar proof) if $Y_n$ is the total magnetization in $\Lambda_n$ from the unique infinite-volume Gibbs measure (rather than from $\Lambda_n$ with free boundary conditions).
\end{remark}

\begin{remark}
	We expect that $\beta_1(d)=\beta_c(d)$ should be valid in Corollary \ref{cor1}. In the Appendix, we prove that when $\beta<\beta_c(d)$, the limiting distribution (or measure) of Lee-Yang zeros has no mass in an arc containing $\exp(i0)$ of the unit circle. This is an indication that $\beta_1(d)$ ought to be $\beta_c(d)$.
\end{remark}

\begin{remark}
	When $d\geq 4$ and $\beta=\beta_c(d)$, one possible candidate for $f_X(x)$ is the symmetric mixture of two independent normal densities with opposite means $\pm \sqrt{2}/2$ and variance~$1/2$.
\end{remark}

\subsection{Discussion about high-dimensional Ising model with periodic boundary conditions}\label{subsec:HighDIsing}
We now consider the Ising model on $\Lambda_n$ at inverse temperature $\beta\geq 0$ with periodic boundary conditions. Let $\langle\cdot\rangle_{\Lambda_n,\beta,p}$ denote its expectation and $\bar{Y}_n:=\sum_{u\in\Lambda_n}\sigma_u$ be its total magnetization, where the overbar is a reminder that we are using periodic boundary conditions. The infinite product representation for the moment generating function given in Lemma \ref{lem:Y_nfactor} below also holds for this case. So we have for any $d\geq1$ and $\beta\geq 0$,
\begin{align}\label{eq:Y_npmgf}
	&\left\langle \exp\left(z\frac{\bar{Y}_n}{\sqrt{\langle \bar{Y}_n^2 \rangle_{\Lambda_n,\beta,p}}}\right)\right\rangle_{\Lambda_n,\beta,p}\nonumber\\
	&\qquad=\exp\left(\frac{z^2}{2}\right)\prod_{j\geq 1}\left[\left(1+\frac{z^2}{\bar{\alpha}^2_{j,n}\langle \bar{Y}_n^2 \rangle_{\Lambda_n,\beta,p}}\right)\exp\left(-\frac{z^2}{\bar{\alpha}_{j,n}^2\langle \bar{Y}_n^2 \rangle_{\Lambda_n,\beta,p}}\right)\right],
\end{align}
where $0<\bar{\alpha}_{1,n}\leq\bar{\alpha}_{2,n}\leq\dots$, $\{\pm i\bar{\alpha}_{j,n}:j\geq 1\}$ are all the zeros (listed with multiplicities) of $\langle \exp(z\bar{Y}_n)\rangle_{\Lambda_n,\beta,p}$, and $\sum_{j\geq 1}1/\bar{\alpha}_{j,n}^2\leq\langle \bar{Y}_n^2\rangle_{\Lambda_n,\beta,p}/2$. Combining \eqref{eq:Y_npmgf} with the inequality $|(1+y)\exp(-y)|\leq\exp(2|y|)$ gives that
\begin{equation}
	\left|\left\langle \exp\left(z\frac{\bar{Y}_n}{\sqrt{\langle \bar{Y}_n^2 \rangle_{\Lambda_n,\beta,p}}}\right)\right\rangle_{\Lambda_n,\beta,p}\right|\leq \exp\left(\frac{3|z|^2}{2}\right) \text{ for any }z\in\mathbb{C}.
\end{equation}
Therefore, $\{\langle \exp(z\bar{Y}_n/\sqrt{\langle \bar{Y}_n^2 \rangle_{\Lambda_n,\beta,p}})\rangle_{\Lambda_n,\beta,p}:n\in\mathbb{N}\}$ is locally uniformly bounded. Suppose that
\begin{equation}
	\bar{Y}_n/\sqrt{\langle \bar{Y}_n^2 \rangle_{\Lambda_n,\beta,p}}\Longrightarrow \bar{Y} \text{ as }n\rightarrow \infty.
\end{equation}
Then we have
\begin{equation}
	\lim_{n\rightarrow\infty}\left\langle \exp\left(it\bar{Y}_n/\sqrt{\langle \bar{Y}_n^2 \rangle_{\Lambda_n,\beta,p}}\right)\right\rangle=\mathbb{E}\exp(it\bar{Y}) \text{ for any }t\in\mathbb{R}.
\end{equation}
So Vitali's theorem (see, e.g., Theorem B.25 of \cite{FV17}) implies that
\begin{equation}\label{eq:Y_npmgfconv}
    \lim_{n\rightarrow\infty}\left\langle \exp\left(z\bar{Y}_n/\sqrt{\langle \bar{Y}_n^2 \rangle_{\Lambda_n,\beta,p}}\right)\right\rangle=\mathbb{E}\exp(z\bar{Y}) \text{ locally uniformly on  }\mathbb{C}.
\end{equation}
Theorem 7 of \cite{NW19} and Theorem \ref{thm:new75} below give that 
\begin{equation}\label{eq:Ypmgf}
	\mathbb{E}\exp(z\bar{Y})=\exp\left(\frac{z^2}{2}\right)\prod_{j\geq 1}\left[\left(1+\frac{z^2}{\bar{a}_j^2}\right)\exp\left(-\frac{z^2}{\bar{a}_j^2}\right)\right],~\forall z\in\mathbb{C},
\end{equation}
where $\{\pm i\bar{a}_j:j\geq 1\}$ are all the zeros of  $\mathbb{E}\exp(z\bar{Y})$ with $0<\bar{a}_1\leq\bar{a}_2\leq\dots$. By Hurwitz's theorem (see Lemma \ref{lem:Hur} below), we have
\begin{equation}\label{eq:alphapcon}
	\lim_{n\rightarrow\infty}\bar{\alpha}_{j,n}\sqrt{\langle \bar{Y}_n^2 \rangle_{\Lambda_n,\beta,p}}=\bar{a}_j \text{ for each }j\geq 1.
\end{equation}

In the rest of this subsection, we focus on $d> 4$ and $\beta=\beta_c:=\beta_c(d)$. Here $d=4$ is the upper critical dimension, and we exclude $d=4$ since the corresponding critical behavior is subtle and various quantities may have logarithmic corrections. By a conditional result of \cite{Pap06}, Theorem 2 there,  it is expected that
\begin{equation}
	\liminf_{n\rightarrow\infty}\bar{\alpha}_{1,n}^2\langle \bar{Y}_n^2 \rangle_{\Lambda_n,\beta_c,p}<\infty
\end{equation}
since otherwise \eqref{eq:Ypmgf} and \eqref{eq:alphapcon} would imply that $\bar{Y}$ is standard normal.

Recall that $u_4(\bar{Y}_n)$ is the fourth cumulant of $\bar{Y}_n$. It is well-known  (see Theorem 5 of \cite{New75} or (5.3) of \cite{Aiz82}) that
\begin{equation}
	0\leq -u_4(\bar{Y}_n)\leq 3\left[\langle \bar{Y}_n^2 \rangle_{\Lambda_n,\beta_c,p}\right]^2
\end{equation}
It is also known (see, e.g., \eqref{eq:Y_nmgfbds}-\eqref{eq:u_4overvar} below for a proof) that if $\bar{Y}$ is not standard normal, then 
\begin{equation}
	\liminf_{n\rightarrow\infty}\frac{-u_4(\bar{Y}_n)}{\left[\langle \bar{Y}_n^2 \rangle_{\Lambda_n,\beta_c,p}\right]^2}>0.
\end{equation}
So it is expected that there exists $C_6\in(0,\infty)$ such that
\begin{equation}
	\lim_{n\rightarrow\infty}\frac{-u_4(\bar{Y}_n)}{\left[\langle \bar{Y}_n^2 \rangle_{\Lambda_n,\beta_c,p}\right]^2}=C_6^{-4}.
\end{equation}

To summarize, we formulate a conjecture relating $\bar{Y}$ to the asymptotic behavior of the $\bar{\alpha}_{j,n}$'s.
\begin{conjecture}\label{conj1}
	For the critical Ising model on $\Lambda_n$ with periodic boundary conditions and $d>4$. The following limits exist
	\begin{equation}
		\bar{a}_j:=\lim_{n\rightarrow\infty}\bar{\alpha}_{j,n}\sqrt{\langle \bar{Y}_n^2 \rangle_{\Lambda_n,\beta_c,p}}=C_6\lim_{n\rightarrow\infty}\bar{\alpha}_{j,n}\left[-u_4(\bar{Y}_n)\right]^{1/4} \text{ for each }j\geq1.
	\end{equation}
Hence, we have
\begin{equation}
	\bar{Y}_n/\sqrt{\langle \bar{Y}_n^2 \rangle_{\Lambda_n,\beta_c,p}}\Longrightarrow \bar{Y} \text{ as }n\rightarrow \infty,
\end{equation}
where $\bar{Y}$ has moment generating function
\begin{equation}
	\mathbb{E}\exp(z\bar{Y})=\exp\left(\frac{z^2}{2}\right)\prod_{j\geq 1}\left[\left(1+\frac{z^2}{\bar{a}_j^2}\right)\exp\left(-\frac{z^2}{\bar{a}_j^2}\right)\right],~\forall z\in\mathbb{C}.
\end{equation}
\end{conjecture}

Comparing Theorem \ref{thm1} and Conjecture \ref{conj1}, it is natural to ask whether
\begin{equation}\label{eq:cwpp}
	\lim_{n\rightarrow\infty}\alpha_{j,n}[-u_4(Y_n)]^{1/4}=\lim_{n\rightarrow\infty}\bar{\alpha}_{j,n}\left[-u_4(\bar{Y}_n)\right]^{1/4} \text{ for each }j
\end{equation}
or at least whether the ratio of these two limits is the same constant for all $j\geq 1$. Even though the behavior of $Y_n$ (under free boundary conditions) is quite different from that of $\bar{Y}_n$ (under periodic boundary conditions), the answer to the above question could be positive.

In Section \ref{sec:Gautran}, we use a Gaussian transform of $X_n$ to obtain a new random variable, $W_n$, whose density is directly related to the moment generating function of $Y_n$. Then we use the asymptotics of $Y_n$ to show that $W_n$ and $X_n$ (after appropriate rescalings) converge along the same subsequence and yield the same limit up to a linear transformation. In Section~\ref{sec:proofs}, we first prove that for any subsequential limiting distribution of the normalized $X_n$, $F_X$, the moment generating  function of $\exp(-bx^2)dF_X(x)$ for each positive $b$ has only pure imaginary zeros (see Proposition \ref{prop:ZLY}); this enables us to finish the proof of Theorem~\ref{thm1} by applying a complete classification result about such distributions from \cite{New76}. In the Appendix, we prove that the limiting distribution of Lee-Yang zeros has no mass in an arc containing $\exp(i0)$ of the unit circle for any $\beta<\beta_c(d)$.

\section{A Gaussian transform and moment generating functions}\label{sec:Gautran}
\subsection{A Gaussian transform}
In order to study the limit of $X_n/\sqrt{\mathbb{E}_{\Lambda_n,\beta}(X_n^2)}$, as a tool, we add to $X_n$ a multiple of a standard normal random variable $N(0,1)$, which is independent of $X_n$. That is 
\begin{equation}\label{eq:Wndef}
	W_n:=X_n+\sqrt{\lambda_n}N(0,1),
\end{equation} 
where $\lambda_n\geq 0$ will be defined later.

To simply the notation, we define
\begin{equation}\label{eq:gamma_ndef}
	\gamma_n:=\frac{1}{2\langle Y_n^2 \rangle_{\Lambda_n,\beta}}.
\end{equation}
Then we have
\begin{equation}\label{eq:RN}
	\frac{d F_{X_n}}{d F_{Y_n}}(x)=\frac{\exp[\gamma_nx^2]}{\left\langle \exp[\gamma_nY_n^2]\right\rangle_{\Lambda_n,\beta}}.
\end{equation}
	
\begin{lemma}
	The random variable $W_n$ has density
	\begin{equation}\label{eq:W_ndensity1}
		f_{W_n}(x)=\frac{\exp[-x^2/(2\lambda_n)]}{\sqrt{2\pi\lambda_n}\langle \exp(\gamma_n Y_n^2)\rangle_{\Lambda_n,\beta}}\int_{-\infty}^{\infty}\exp\left(\frac{xy}{\lambda_n}\right)\exp\left[\left(\gamma_n-\frac{1}{2\lambda_n}\right)y^2\right] dF_{Y_n}(y).
	\end{equation}
\end{lemma}
\begin{proof}
	This follows from \eqref{eq:Wndef}, \eqref{eq:RN} and the convolution formula (see, e.g., Theorem 2.1.16 of \cite{Dur19}).
\end{proof}

In order to get rid of the exponential function with $y^2$ term in the integrand of \eqref{eq:W_ndensity1}, we fix $\lambda_n$ once and for all:
\begin{equation}\label{eq:lambda_ndef}
	\lambda_n:=\langle Y_n^2 \rangle_{\Lambda_n,\beta}.
\end{equation}
Then we have
\begin{equation}\label{eq:W_ndensity}
		f_{W_n}(x)=\frac{\exp[-x^2/(2\lambda_n)]}{\sqrt{2\pi\lambda_n} \langle \exp(\gamma_n Y_n^2)\rangle_{\Lambda_n,\beta}}\left\langle \exp[xY_n/\lambda_n] \right \rangle_{\Lambda_n,\beta}.
\end{equation}
	
\subsection{Random variables of Lee-Yang type}
A random variable $X$ is said to be of LY type if
\begin{enumerate}[(a)]
	\item $X\overset{d}{=}-X$ (i.e., $\mathbb{P}(X\in B)=\mathbb{P}(-X\in B)$ for any Borel set $B$),
	\item $\mathbb{E}\exp(DX^2)<\infty$ for some $D>0$,
	\item $\mathbb{E}\exp(zX)$ has only pure imaginary zeros.
\end{enumerate}

The following theorem from \cite{New75} says that the moment generating function of a LY type random variable has a nice factorization formula. 
\begin{theorem}[\cite{New75}]\label{thm:new75}
	If $X$ is of LY type, then
	\begin{equation}
		\mathbb{E}\exp(zX)=\exp(bz^2)\prod_{j\geq 1}\left[1+\frac{z^2}{\alpha_j^2}\right],~z\in\mathbb{C},
	\end{equation}
for some $b\geq 0$ and $0<\alpha_1\leq\alpha_2\leq\dots$, with $\sum_{j\geq 1}1/\alpha_j^2<\infty$; $\{\alpha_j:j\geq 1\}$ may be empty, finite or infinite. Moreover,
\begin{align}
	&u_2(X)=2\left(b+\sum_{j\geq 1}\frac{1}{\alpha_j^2}\right),\label{eq:bdef}\\
	&u_{2m}(X)=(-1)^{m-1}\frac{(2m)!}{m}\sum_{j\geq 1}\frac{1}{\alpha_j^{2m}},~ m\geq 2 \text{ and } m\in\mathbb{N}.\label{eq:cumdef}
\end{align}
\end{theorem}

\subsection{Asymptotics of $Y_n$ and $W_n$}
A direct application of Theorem \ref{thm:new75} expressing $b$ using \eqref{eq:bdef} gives
\begin{lemma} \label{lem:Y_nfactor}
	For any $d\geq 1$ and $\beta\geq 0$, we have that for any $z\in\mathbb{C}$, 
	\begin{equation}\label{eq:Y_nmgffactor}
		\langle \exp(zY_n)\rangle_{\Lambda_n,\beta}=\exp\left[z^2\langle Y_n^2\rangle_{\Lambda_n,\beta}/2\right]\prod_{j\geq 1}\left[\left(1+\frac{z^2}{\alpha^2_{j,n}}\right)\exp\left(-\frac{z^2}{\alpha_{j,n}^2}\right)\right],
	\end{equation}
where $0<\alpha_{1,n}\leq\alpha_{2,n}\leq\dots$, $\{\pm i\alpha_{j,n}:j\geq 1\}$ are all zeros (listed according to their multiplicities) of $\langle \exp(zY_n)\rangle_{\Lambda_n,\beta}$, and $\sum_{j\geq 1}1/\alpha_{j,n}^2\leq\langle Y_n^2\rangle_{\Lambda_n,\beta}/2$.
\end{lemma}

We define for $n\in\mathbb{N}$,
\begin{align}
	&c_n:=\frac{\sqrt{\langle Y_n^2 \rangle_{\Lambda_n,\beta}}}{\sqrt{2\pi}\langle\exp[Y_n^2/(2\langle Y_n^2\rangle_{\Lambda_n,\beta})]\rangle_{\Lambda_n,\beta}}\left[\frac{4!}{-u_4(Y_n)}\right]^{1/4},\label{eq:c_ndef}\\
	&d_n:=\frac{1}{\langle Y_n^2 \rangle_{\Lambda_n,\beta}}\left[\frac{-u_4(Y_n)}{4!}\right]^{1/4}.\label{eq:d_ndef}
\end{align}
Then we have 
\begin{lemma}\label{lem:d_nW_n}
For any $d\geq 1$ and $\beta\geq 0$, the family of random variables $\{d_nW_n:n\in\mathbb{N}\}$ is tight --- more precisely,
	\begin{equation}
		c_n\exp(-x^4)\leq f_{d_nW_n}(x)\leq\frac{c_n}{1+x^4/3} \text{ for any }x\in\mathbb{R}\label{eq:d_nW_ndesitybd}
	\end{equation}
with the constants $c_n$ satisfying the bounds
\begin{equation}
		\left[\int_{-\infty}^{\infty}\frac{1}{1+x^4/3}dx\right]^{-1}\leq c_n \leq \left[\int_{-\infty}^{\infty}\exp(-x^4)dx\right]^{-1} \text{ for any }n\in\mathbb{N}.\label{eq:c_nbd}
\end{equation}
\end{lemma}
\begin{proof}
A direct computation using \eqref{eq:W_ndensity}, Lemma \ref{lem:Y_nfactor}, \eqref{eq:c_ndef} and \eqref{eq:d_ndef} gives that
	\begin{equation}\label{eq:d_nW_ndensity}
		f_{d_nW_n}(x)=\frac{1}{d_n}f_{W_n}\left(\frac{x}{d_n}\right)=c_n\prod_{j\geq 1}\left[1+\frac{x^2}{\alpha_{j,n}^2[-u_4(Y_n)/4!]^{1/2}}\right]\exp\left[-\frac{x^2}{\alpha_{j,n}^2[-u_4(Y_n)/4!]^{1/2}}\right].
	\end{equation}
Now \eqref{eq:d_nW_ndesitybd} follows by applying the two inequalities below to the exponential term in \eqref{eq:d_nW_ndensity}
\begin{equation}\label{eq:twoine}
	\exp(-y^2/2)\leq(1+y)\exp(-y)\leq\frac{1}{1+y^2/6} \text{ for any }y\geq0.
\end{equation}
Clearly, \eqref{eq:d_nW_ndesitybd} implies that $\{d_nW_n:n\in\mathbb{N}\}$ is tight and \eqref{eq:c_nbd}.
\end{proof}

The following lemma contains some important properties of $X_n$ and $Y_n$.
\begin{lemma}\label{lem:X_nY_n}
One has
	\begin{equation}\label{eq:Y_nCLT}
		\frac{Y_n}{\sqrt{\langle Y_n^2 \rangle_{\Lambda_n,\beta}}}\Longrightarrow N(0,1),
	\end{equation}
	where $N(0,1)$ is a standard normal random variable in the following two situations: (i) any $d\geq 1$ and any $\beta\in[0,\beta_c(d))$, (ii) any $d\geq 4$ and $\beta=\beta_c(d)$. Furthermore, in those two situations, one has
	\begin{equation}\label{eq:Y_nmgf}
		\left|\langle \exp(zY_n)\rangle_{\Lambda_n,\beta}\right|\leq \exp\left[|z|^2\langle Y_n^2 \rangle_{\Lambda_n,\beta}/2\right],~\forall z\in\mathbb{C};
	\end{equation}
	\begin{equation}\label{eq:Y_noverX_n}
		\lim_{n\rightarrow\infty}\frac{\langle Y_n^2 \rangle_{\Lambda_n,\beta}}{\mathbb{E}_{\Lambda_n,\beta}(X_n^2)}=0.
	\end{equation}
\end{lemma}
\begin{proof}
	For $\beta<\beta_c(d)$, if $Y_n$ were from a translation invariant system (see the second part of Remark \ref{rem:Y_ngen}), then \eqref{eq:Y_nCLT} would be a standard result (see, e.g., \cite{New80}). For our current $Y_n$, we note that Lemma~\ref{lem:Y_nfactor} and \eqref{eq:twoine} imply that for any $x\in\mathbb{R}$
	\begin{equation}\label{eq:Y_nmgfbds}
		\exp(x^2/2)\exp\left[\frac{x^4 u_4(Y_n)}{24\lambda_n^2}\right]\leq\langle\exp[xY_n/\sqrt{\lambda_n}]\rangle_{\Lambda_n,\beta}\leq \exp(x^2/2)\frac{1}{1-x^4u_4(Y_n)/(72\lambda_n^2)}.
	\end{equation}
So \eqref{eq:Y_nCLT} follows if one can show
\begin{equation}\label{eq:u_4overvar}
	\lim_{n\rightarrow\infty}\frac{u_4(Y_n)}{\lambda_n^2}=0;
\end{equation}
This follows from two observations. On the one hand, Aizenman's inequality for finite systems (Proposition 5.3 of \cite{Aiz82}) and exponential decay of two-point functions when $\beta<\beta_c(d)$ (see \cite{ABF87}) imply that $u_4(Y_n)$ is at most of order $|\Lambda_n|$. On the other hand, the trivial exponential lower bound obtained using the Edwards-Sokal coupling (see, e.g., Theorems 1.16 and 3.1 of \cite{Gri06}),
	\begin{equation}
		\langle \sigma_x\sigma_y\rangle_{\Lambda_n,\beta}\geq\left(\frac{1-\exp(-2\beta)}{1+\exp(-2\beta)}\right)^{\|x-y\|_1},~\forall x,y\in\Lambda_n,
	\end{equation}
implies that $\lambda_n$ is at least of order $|\Lambda_n|$. For $d\geq 4$ and $\beta=\beta_c(d)$, \eqref{eq:Y_nCLT} follows from \cite{Aiz82,Fro82,ADC21}. The inequality \eqref{eq:Y_nmgf} follows from \eqref{eq:Y_nmgffactor}. For \eqref{eq:Y_noverX_n}, we have
	\begin{equation}\label{eq:Y_noverX_n1}
		\frac{\langle Y_n^2 \rangle_{\Lambda_n,\beta}}{\mathbb{E}_{\Lambda_n,\beta}(X_n^2)}=\frac{\langle Y_n^2 \rangle_{\Lambda_n,\beta}\left\langle\exp[Y_n^2/(2\langle Y_n^2 \rangle_{\Lambda_n,\beta})]\right\rangle_{\Lambda_n,\beta}}{\left\langle Y_n^2\exp[Y_n^2/(2\langle Y_n^2 \rangle_{\Lambda_n,\beta})]\right\rangle_{\Lambda_n,\beta}}:=\frac{\langle\exp(V_n/2)\rangle_{\Lambda_n,\beta}}{\langle V_n\exp(V_n/2)\rangle_{\Lambda_n,\beta}},
	\end{equation}
	where $V_n:=Y_n^2/\langle Y_n^2 \rangle_{\Lambda_n,\beta}$. \eqref{eq:Y_nCLT} and \eqref{eq:Y_nmgf} imply (using the convergence of moments and the inequality $\exp(x)\geq \sum_{k=1}^M x^k/k!$ for any $x\geq 0$ and $M\in\mathbb{N}$) that
	\begin{equation}\label{eq:Vnmgf}
		\lim_{n\rightarrow \infty}\langle \exp(V_n/2)\rangle_{\Lambda_n,\beta}=\infty.
	\end{equation}
	For any $M>0$, we get from \eqref{eq:Y_noverX_n1} that
	\begin{equation}
		\frac{\langle Y_n^2 \rangle_{\Lambda_n,\beta}}{\mathbb{E}_{\Lambda_n,\beta}(X_n^2)}\leq \frac{\langle 1[V_n\geq M]\exp(V_n/2)\rangle_{\Lambda_n,\beta}+\langle 1[V_n< M]\exp(V_n/2)\rangle_{\Lambda_n,\beta}}{M\langle 1[V_n\geq M]\exp(V_n/2)\rangle_{\Lambda_n,\beta}},
	\end{equation}
	where $1[\cdot]$ is the indicator function. This combined with \eqref{eq:Vnmgf} implies that
	\begin{equation}
		\limsup_{n\rightarrow\infty}\frac{\langle Y_n^2 \rangle_{\Lambda_n,\beta}}{\mathbb{E}_{\Lambda_n,\beta}(X_n^2)}=0,
	\end{equation}
	which completes the proof of \eqref{eq:Y_noverX_n}.
\end{proof}

We define
\begin{equation}\label{eq:W_ndef}
	\tilde{W}_n:=\frac{W_n}{\sqrt{\mathbb{E}_{\Lambda_n,\beta}(X_n^2)}}=\frac{X_n+\sqrt{\lambda_n}N(0,1)}{\sqrt{\mathbb{E}_{\Lambda_n,\beta}(X_n^2)}}.
\end{equation}
Then by \eqref{eq:Y_noverX_n}, we have
\begin{equation}\label{eq:W_nsmlimit}
	\mathbb{E}(\tilde{W}_n^2)=\frac{\mathbb{E}_{\Lambda_n,\beta}X_n^2+\langle Y_n^2\rangle_{\Lambda_n,\beta}}{\mathbb{E}_{\Lambda_n,\beta}(X_n^2)}\rightarrow 1 \text{ as }n\rightarrow\infty.
\end{equation}
So $\{\tilde{W}_n:n\in\mathbb{N}\}$ is tight. The following lemma compares the three families of random variables: $\{d_nW_n:n\in\mathbb{N}\}$, $\{\tilde{W}_n:n\in\mathbb{N}\}$, and $\{X_n/\sqrt{\mathbb{E}_{\Lambda_n,\beta,f}(X_n^2)}:n\in\mathbb{N}\}$.
\begin{lemma}\label{lem:3sequences}
Suppose that there exists a subsequence of $\mathbb{N}$, $\{n_k:k\in\mathbb{N}\}$, such that
	\begin{equation}
		d_{n_k}W_{n_k}\Longrightarrow W \text{ as }k\rightarrow\infty \text{ for some random variable }W,
	\end{equation} 
then
\begin{equation}
	\tilde{W}_{n_k}\Longrightarrow C_7W \text{ and }X_{n_k}/\sqrt{\mathbb{E}_{\Lambda_{n_k},\beta}(X_{n_k}^2)}\Longrightarrow C_7W \text{ as }k\rightarrow\infty,
\end{equation}
where $C_7:=\sqrt{1/\mathbb{E}(W^2)}\in(0,\infty)$.
\end{lemma}
\begin{remark}\label{rem:3sequences}
	Lemma \ref{lem:3sequences} and the proof of Theorem \ref{thm1} imply the following stronger result: if one of the three families of random variables $\{d_nW_n:n\in\mathbb{N}\}$, $\{\tilde{W}_n:n\in\mathbb{N}\}$ and $\{X_n/\sqrt{\mathbb{E}_{\Lambda_n,\beta}(X_n^2)}:n\in\mathbb{N}\}$ has a convergent (in distribution) subsequence, say along $\{n_k:k\in\mathbb{N}\}$, then the other two families also converge (in distribution) along the same subsequence with their limits equal to the former one up to a linear transformation; and all subsequential limits are not trivial (i.e., not the Dirac point measure at $0$).
\end{remark}
\begin{proof}
	$\{X_n/\sqrt{\mathbb{E}_{\Lambda_n,\beta}(X_n^2)}:n\in\mathbb{N}\}$ is tight since the second moments are uniformly bounded by $1$. The converging together lemma (see p. 108 of \cite{Dur19}) and \eqref{eq:W_ndef}, \eqref{eq:Y_noverX_n} imply that $\{\tilde{W}_n:n\in\mathbb{N}\}$ and $\{X_n/\sqrt{\mathbb{E}_{\Lambda_n,\beta}(X_n^2)}:n\in\mathbb{N}\}$ have the same convergent subsequences and same corresponding limits. By the tightness of $\{\tilde{W}_n:n\in\mathbb{N}\}$, we can find a subsequence of $\{n_k:k\in\mathbb{N}\}$, $\{m_l:l\in\mathbb{N}\}$, such that
\begin{equation}
	\tilde{W}_{m_l}\Longrightarrow C_7W \text{ for some }C_7\geq 0\text{ as }l\rightarrow\infty.
\end{equation} 
By the GHS inequality \cite{GHS70} and \eqref{eq:Y_noverX_n}, we have that for all large $n$ and any $x\in\mathbb{R}$,
\begin{equation}
	\mathbb{E}\exp[x\tilde{W}_n]=\exp\left[x^2\lambda_n/\left(2\mathbb{E}_{\Lambda_n,\beta}(X_n^2)\right)\right]\mathbb{E}_{\Lambda_n,\beta}\exp\left(xX_n/\sqrt{\mathbb{E}_{\Lambda_n,\beta}(X_n^2)}\right)\leq \exp(x^2).
\end{equation}
So we have by \eqref{eq:W_nsmlimit} that
\begin{equation}\label{eq:tildeW_nsmlimit}
	1=\lim_{l\rightarrow\infty} \mathbb{E}\left(\tilde{W}_{m_l}\right)^2=\mathbb{E}(C_7W)^2.
\end{equation}
Therefore,
\begin{equation}
	C_7:=\sqrt{1/\mathbb{E}(W^2)}
\end{equation}
and all subsequential limit of $\{\tilde{W}_{n_k}:k\geq 1\}$ are the same and thus
\begin{equation}
	\tilde{W}_{n_k}\Longrightarrow C_7W \text{ as }k\rightarrow \infty.
\end{equation}
We note that $W\neq 0$ follows from \eqref{eq:d_nW_ndesitybd}, \eqref{eq:c_nbd}.
\end{proof}

\section{Proofs of the main results}\label{sec:proofs}
\subsection{A proposition}
The following proposition is a key step towards the proof of Theorem \ref{thm1}.
\begin{proposition}\label{prop:ZLY}
	Suppose that
	\begin{equation}
	     X_n/\sqrt{\mathbb{E}_{\Lambda_n,\beta}(X_n^2)}\Longrightarrow X \text{ as }n\rightarrow\infty \text{ for some random variable }X.
	\end{equation}
For any fixed $b>0$, let $Z$ be a random variable with distribution function $F_{Z}$ given by
\begin{equation}\label{eq:RNZ_noverW}
	\frac{d F_{Z}}{d F_{X}}(x)=\frac{\exp(-bx^2)}{\mathbb{E}\exp[-bX^2]},
\end{equation}
where $F_X$ is the distribution function of $X$. Then the random variable $Z$ is of LY type.
\end{proposition}
\begin{proof}
	To simply the notation, we denote
	\begin{equation}
		\tilde{X}_n:=X_n/\sqrt{\mathbb{E}_{\Lambda_n,\beta}(X_n^2)}.
	\end{equation}
Define a sequence of random variable $Z_n$ (or more precisely, a sequence of distribution functions $F_{Z_n}$) by
\begin{equation}\label{eq:RNZ_noverX_n}
	\frac{d F_{Z_n}}{d F_{\tilde{X}_n}}(x)=\frac{\exp(-bx^2)}{\mathbb{E}_{\Lambda_n,\beta}\exp[-b\tilde{X}_n^2]}.
\end{equation}
Then for any continuous and bounded function $g$, by the continuous mapping theorem (Theorem 3.2.10 of \cite{Dur19}), we have
\begin{equation}
	\mathbb{E}g(Z_n)=\frac{\mathbb{E}_{\Lambda_n,\beta}\left[g(\tilde{X}_n)\exp[-b\tilde{X}_n^2]\right]}{\mathbb{E}_{\Lambda_n,\beta}\exp[-b\tilde{X}_n^2]}\rightarrow \frac{\mathbb{E}\left[g(X)\exp[-bX^2]\right]}{\mathbb{E}\exp[-bX^2]} \text{ as }n\rightarrow\infty.
\end{equation}
Therefore,
\begin{equation}\label{eq:Z_ntoZ}
	Z_n\Longrightarrow Z \text{ as }n\rightarrow\infty.
\end{equation}
Equation \eqref{eq:RN} implies that
\begin{equation}\label{eq:RNtildeX_noverY_n}
	d F_{\tilde{X}_n}(x)=d F_{X_n}\left(x\sqrt{\mathbb{E}_{\Lambda_n,\beta}(X_n^2)}\right)=\frac{\exp[\gamma_n x^2\mathbb{E}_{\Lambda_n,\beta}(X_n^2)]}{\langle\exp[\gamma_n Y_n^2]\rangle_{\Lambda_n,\beta}}d F_{Y_n}\left(x\sqrt{\mathbb{E}_{\Lambda_n,\beta}(X_n^2)}\right).
\end{equation}
Combining \eqref{eq:RNZ_noverX_n} and \eqref{eq:RNtildeX_noverY_n}, we get
\begin{equation}\label{eq:RNZ_noverY_n}
	d F_{Z_n}(x)=\frac{\exp[\hat{\gamma}_n x^2\mathbb{E}_{\Lambda_n,\beta}(X_n^2)]}{\mathbb{E}_{\Lambda_n,\beta}\exp[-b\tilde{X}_n^2]\langle\exp[\gamma_n Y_n^2]\rangle_{\Lambda_n,\beta}}d F_{Y_n}\left(x\sqrt{\mathbb{E}_{\Lambda_n,\beta}(X_n^2)}\right),
\end{equation}
where 
\begin{equation}\label{eq:hatgammadef}
	\hat{\gamma}_n:=\gamma_n\left(1-\frac{b}{\gamma_n\mathbb{E}_{\Lambda_n,\beta}(X_n^2)}\right).
\end{equation}
Let $\hat{X}_n$ be the total magnetization for the perturbed Ising model with Hamiltonian
\begin{equation}\label{eq:hatH}
	\hat{H}_{\Lambda_n,\beta}(\sigma):=-\beta\sum_{\{u,v\}}\sigma_u\sigma_y-\hat{\gamma}_nY^2_n(\sigma),\qquad \sigma\in\{-1,+1\}^{\Lambda_n}.
\end{equation}
Then
\begin{equation}
	\frac{d F_{\hat{X}_n}}{d F_{Y_n}}(x)=\frac{\exp[\hat{\gamma}_n x^2]}{\langle \exp[\hat{\gamma}_n Y_n^2]\rangle_{\Lambda_n,\beta}}.
\end{equation}
Hence
\begin{equation}\label{eq:RNhatX_noverY_n}
	d F_{\hat{X}_n}\left(x\sqrt{\mathbb{E}_{\Lambda_n,\beta}(X_n^2)}\right)=\frac{\exp[\hat{\gamma}_n x^2\mathbb{E}_{\Lambda_n,\beta}(X_n^2)]}{\langle\exp[\hat{\gamma}_n Y_n^2]\rangle_{\Lambda_n,\beta}}d F_{Y_n}\left(x\sqrt{\mathbb{E}_{\Lambda_n,\beta}(X_n^2)}\right)
\end{equation}
Equations \eqref{eq:RNZ_noverY_n} and \eqref{eq:RNhatX_noverY_n} imply that
\begin{equation}\label{eq:Z_n=hatX_n}
	Z_n\overset{d}{=}\frac{\hat{X}_n}{\sqrt{\mathbb{E}_{\Lambda_n,\beta}(X_n^2)}},
\end{equation}
(where $\overset{d}{=}$ denotes equal in distribution) and
\begin{equation}
	\mathbb{E}_{\Lambda_n,\beta}\exp[-b\tilde{X}_n^2]\langle\exp[\gamma_n Y_n^2]\rangle_{\Lambda_n,\beta}=\langle\exp[\hat{\gamma}_n Y_n^2]\rangle_{\Lambda_n,\beta}.
\end{equation}
From the definition of $\hat{\gamma}_n$ in \eqref{eq:hatgammadef}, $\gamma_n$ in \eqref{eq:gamma_ndef}, and \eqref{eq:Y_noverX_n}, we know that
\begin{equation}
	\lim_{n\rightarrow\infty}\frac{\hat{\gamma}_n}{\gamma_n}=1.
\end{equation}
So the interactions in \eqref{eq:hatH} are ferromagnetic for all large $n$. Therefore, by the Lee-Yang theorem \cite{LY52}, the moment generating function of $\hat{X}_n$ has only pure imaginary zeros. So \eqref{eq:Z_n=hatX_n} says that each $Z_n$ is of LY type, and thus $Z$ is of LY type by \eqref{eq:Z_ntoZ} and Theorem 7 of~\cite{NW19}.
\end{proof}

\subsection{Proof of Theorem \ref{thm1}}
For ease of reference, we state without proof the following lemma since the proof follows directly from Hurwitz's theorem (see, e.g., p. 4 of \cite{Mar66}).
\begin{lemma}\label{lem:Hur}
	Let $f_n$ be a sequence of analytic functions on $\mathbb{C}$ such that all zeros (listed according to their multiplicities) of $f_n$ are $\{\pm i\alpha_{j,n}:j\geq 1\}$ with $0<\alpha_{1,n}\leq\alpha_{2,n}\leq\dots$. Suppose that $f_n$ converge uniformly on compact subsets of $\mathbb{C}$ to an analytic function $f$, and all zeros (listed according to their multiplicities) of $f$ are $\{\pm ia_j:j\geq 1\}$ with $0<a_1\leq a_2\leq\dots$. Then we have
	\[\lim_{n\rightarrow\infty}\alpha_{j,n}=a_j \text{ for each }j\geq 1.\]
\end{lemma}

We are ready to prove Theorem \ref{thm1}.
\begin{proof}[Proof of Theorem \ref{thm1}]
	By the assumption of the theorem, we have
	\begin{equation}
		X_{n_k}/\sqrt{\mathbb{E}_{\Lambda_{n_k},\beta}(X_{n_k}^2)}\Longrightarrow X \text{ as }k\rightarrow\infty.
	\end{equation}
The converging together lemma (see p. 108 of \cite{Dur19}) and \eqref{eq:W_ndef}, \eqref{eq:Y_noverX_n}, together with \eqref{eq:lambda_ndef}, imply that $\{\tilde{W}_n:n\in\mathbb{N}\}$ and $\{X_n/\sqrt{\mathbb{E}_{\Lambda_n,\beta}(X_n^2)}:n\in\mathbb{N}\}$ have the same convergent subsequences and same corresponding limits.  So we have
\begin{equation}\label{eq:tildeW_n_klimit}
	\tilde{W}_{n_k}\Longrightarrow X \text{ as }k\rightarrow\infty.
\end{equation}
Lemma \ref{lem:d_nW_n} implies tightness and non-triviality of any subsequential limit; so there is a subsequence of $\{n_k:k\in\mathbb{N}\}$, $\{m_l:l\in\mathbb{N}\}$, and a random variable $W$ such that
\begin{equation}\label{eq:W_m_llimit}
	d_{m_l}W_{m_l}\Longrightarrow W:=C_8X \text{ as }l\rightarrow\infty \text{ for some } C_8\in(0,\infty).
\end{equation}
Let $F_W$ be the distribution function of $W$. Proposition \ref{prop:ZLY} implies that Theorem 2 of \cite{New76}  can be applied to $W$ and so that  either
\begin{equation}\label{eq:F_Wdelta}
	d F_W(x)=\frac{1}{2}\delta_{-x_0}+\frac{1}{2}\delta_{x_0} \text{ for some }x_0\in\mathbb{R},
\end{equation}
or $F_W$ has a density of the form
\begin{equation}\label{eq:F_Wdensity}
	f_W(x)=Kx^{2m}\exp[-\kappa_1x^4-\kappa_2x^2]\prod_{j\geq 1}\left[\left(1+\frac{x^2}{a_j^2}\right)\exp\left(-\frac{x^2}{a_j^2}\right)\right],
\end{equation}
where $K>0$, $m=0,1,2,\dots$, $0<a_1\leq a_2\leq\dots$, $\sum_{j\geq 1}1/a_j^4<\infty$, $\kappa_1>0$ and $\kappa_2\in\mathbb{R}$ (or $\kappa_1=0$ and $\kappa_2+\sum_{j\geq 1}1/a_j^2>0$); $\{a_j:j\geq 1\}$ may be empty, finite or infinite and the condition $\kappa_2+\sum_{j\geq 1}1/a_j^2>0$ is considered to be satisfied if $\sum_{j\geq 1}1/a_j^2=\infty$.

From \eqref{eq:d_nW_ndensity}, we know that for each fixed $n\in\mathbb{N}$,
\begin{equation}
	f_{d_nW_n}(x) \text{ is decreasing on }[0,\infty)
\end{equation}
since $(1+y)\exp(-y)$ is so. Hence,
\begin{equation}\label{eq:Wdistcomp}
	\mathbb{P}(W\in[0,a])=\lim_{l\rightarrow\infty}\mathbb{P}(d_{m_l}W_{m_l}\in[0,a])\geq\lim_{l\rightarrow\infty}\mathbb{P}(d_{m_l}W_{m_l}\in[a,2a])=	\mathbb{P}(W\in[a,2a])
\end{equation}
for any $a>0$ and $a,2a$ are points of continuity of $F_W$. Note that $W\neq 0$ (see Lemmas \ref{lem:d_nW_n} and \ref{lem:3sequences}, and Remark \ref{rem:3sequences}). So $F_W$ can not be of the form~\eqref{eq:F_Wdelta} because \eqref{eq:F_Wdelta} would not satisfy \eqref{eq:Wdistcomp} for all $a>0$. That means $F_W$ has a density of the form \eqref{eq:F_Wdensity}. In particular, this implies that the density of $W$ evaluated at $0$ is
\begin{equation}\label{eq:C_8Xdensityat0}
	f_{W}(0)=\lim_{\epsilon\downarrow 0}\frac{\lim_{l\rightarrow\infty}\int_0^{\epsilon}f_{d_{m_l}W_{m_l}}(x)dx}{\epsilon}.
\end{equation}
Combining \eqref{eq:C_8Xdensityat0} with Lemma \ref{lem:d_nW_n}, we get that there exists a $c\in(0,\infty)$ such that the following limit exists:
\begin{equation}\label{eq:cdef}
	\lim_{l\rightarrow\infty}c_{m_l}=c:=f_W(0).
\end{equation}
Clearly, $f_W(0)>0$ implies that $m=0$ in \eqref{eq:F_Wdensity}. Note that by \eqref{eq:cumdef}
\begin{equation}\label{eq:alpha_jlowerbd}
	\alpha_{j,n}^2\left[-u_4(Y_n)/4!\right]^{1/2}=\left[\frac{1}{2}\sum_{k\geq 1}\frac{\alpha_{j,n}^4}{\alpha_{k,n}^4}\right]^{1/2}\geq \left(\frac{j}{2}\right)^{1/2} \text{ for each }j\geq 1.
\end{equation}
So \eqref{eq:d_nW_ndensity}, the Taylor series for $\ln(1+y)$ when $|y|<1$, and Fubini's theorem give
\begin{align}
	f_{d_nW_{n}}(x)&=c_n\exp\left\{\sum_{j\geq 1}\left[\ln\left(1+\frac{x^2}{\alpha_{j,n}^2\left[-u_4(Y_n)/4!\right]^{1/2}}\right)-\frac{x^2}{\alpha_{j,n}^2\left[-u_4(Y_n)/4!\right]^{1/2}}\right]\right\}\nonumber\\
	&=c_n\exp\left\{\sum_{j\geq1}\sum_{k=2}^{\infty}\frac{(-1)^{k-1}x^{2k}}{k\alpha_{j,n}^{2k}\left[-u_4(Y_n)/4!\right]^{k/2}}\right\}\nonumber\\
	&=c_n\exp\left\{\sum_{k=2}^{\infty}\frac{(-1)^{k-1}x^{2k}\sum_{j\geq 1}\alpha_{j,n}^{-2k}}{k\left[-u_4(Y_n)/4!\right]^{k/2}}\right\}, ~\forall x\in\mathbb{R} \text{ and }|x|<\left(\frac{1}{2}\right)^{1/4}.\label{eq:d_nW_ndensitynear0}
\end{align}
Since
\begin{equation}
	\left(\sum_{j\geq 1}\frac{1}{\alpha_{j,n}^{2k}}\right)^{1/(2k)}\leq\left(\sum_{j\geq 1}\frac{1}{\alpha_{j,n}^{4}}\right)^{1/4}=\left(\frac{-2u_4(Y_n)}{4!}\right)^{1/4} \text{ for each }k\geq 2,
\end{equation}
we have
\begin{equation}
	0\leq\frac{\sum_{j\geq 1}\alpha_{j,n}^{-2k}}{\left[-u_4(Y_n)/4!\right]^{k/2}}\leq 2^{k/2} \text{ for each }k\geq 2.
\end{equation}
Therefore, by a standard diagonalization argument, we can extract a subsequence of $\{m_l: l\in\mathbb{N}\}$, $\{p_q: q\in\mathbb{N}\}$, such that the following limits exist
\begin{equation}\label{eq:b_kdef}
	\lim_{q\rightarrow\infty}\frac{\sum_{j\geq 1}\alpha_{j,p_q}^{-2k}}{\left[-u_4(Y_{p_q})/4!\right]^{k/2}}:=b_k\in[0,2^{k/2}] \text{ for all }k\geq 2.
\end{equation}
Combining \eqref{eq:cdef}, \eqref{eq:d_nW_ndensitynear0} and \eqref{eq:b_kdef}, we get
\begin{equation}\label{eq:f_p_qnear0}
	\lim_{q\rightarrow\infty}f_{d_{p_q}W_{p_q}}(x)=c\exp\left\{\sum_{k=2}^{\infty}\frac{(-1)^{k-1}b_kx^{2k}}{k}\right\},~\forall x\in\mathbb{R} \text{ and }|x|<\left(\frac{1}{2}\right)^{1/4}.
\end{equation}
So Lemma \ref{lem:d_nW_n}, \eqref{eq:f_p_qnear0}, the dominated convergence theorem and \eqref{eq:W_m_llimit} imply that for any $-(1/2)^{1/4}<y_1<y_2<(1/2)^{1/4}$,
\begin{equation}
	\lim_{q\rightarrow\infty}\int_{y_1}^{y_2} f_{d_{p_q}W_{p_q}}(x)dx=c\int_{y_1}^{y_2}\exp\left\{\sum_{k=2}^{\infty}\frac{(-1)^{k-1}b_kx^{2k}}{k}\right\}dx=\int_{y_1}^{y_2}f_W(x)dx.
\end{equation}
Hence (noting that $f_W(x)$ is continuous, and the RHS of \eqref{eq:f_p_qnear0} is also continuous in $x$)
\begin{equation}\label{eq:f_Wsmallx}
	f_W(x)=c\exp\left\{\sum_{k=2}^{\infty}\frac{(-1)^{k-1}b_kx^{2k}}{k}\right\},~\forall x\in\mathbb{R} \text{ and }|x|<\left(\frac{1}{2}\right)^{1/4}.
\end{equation}
It follows from \eqref{eq:d_nW_ndensity}, \eqref{eq:c_nbd} and \eqref{eq:alpha_jlowerbd} that $\{f_{d_{n}W_{n}}:n\in\mathbb{N}\}$ is uniformly bounded as $n$ varies in each bounded subset of $\mathbb{C}$. Then Vitali's theorem and \eqref{eq:f_p_qnear0}, \eqref{eq:f_Wsmallx} imply that
\begin{equation}\label{eq:f_p_qconv}
	\lim_{q\rightarrow\infty}f_{d_{p_q}W_{p_q}}(x)=f_{W}(x) \text{ uniformly on compact subsets of }\mathbb{C}.
\end{equation}
From \eqref{eq:d_nW_ndensity}, we know all zeros of $f_{d_{p_q}W_{p_q}}$ are
\begin{equation}
	\{\pm i\alpha_{j,p_q}\left[-u_4(Y_{p_q})/4!\right]^{1/4}: j\geq 1\}\text{ with }0<\alpha_{1,p_q}\leq\alpha_{2,p_q}\leq\dots;
\end{equation}
from \eqref{eq:F_Wdensity} and the fact that $m=0$, we know all zeros of $f_W$ are
\begin{equation}
	\{\pm ia_j:j\geq 1\}\text{ with } 0<a_1\leq a_2\leq\dots.
\end{equation}
Consequently, \eqref{eq:f_p_qconv} and Hurwitz's theorem (see Lemma \ref{lem:Hur}) imply that
\begin{equation}\label{eq:alpha_jconv}
	\lim_{q\rightarrow\infty}\alpha_{j,p_q}\left[-u_4(Y_{p_q})/4!\right]^{1/4}=a_j \text{ for each }j\geq 1.
\end{equation}
So \eqref{eq:alpha_jlowerbd}, \eqref{eq:b_kdef}, \eqref{eq:alpha_jconv} and the dominated convergence theorem give that
\begin{equation}
	b_k=\sum_{j\geq 1}\frac{1}{a_j^{2k}} \text{ for each }k>2.
\end{equation}
Combining this with \eqref{eq:f_Wsmallx} gives that
\begin{equation}\label{eq:f_Wsmallx1}
	f_W(x)=c\exp\left\{-x^4+\sum_{k=3}^{\infty}\frac{(-1)^{k-1}x^{2k}\sum_{j\geq1}a_j^{-2k}}{k}\right\},~\forall x\in\mathbb{R} \text{ and }|x|<\left(\frac{1}{2}\right)^{1/4},
\end{equation}
where we have used $b_2=2$.  By \eqref{eq:alpha_jlowerbd} and \eqref{eq:alpha_jconv}, we have
\begin{equation}
	a_j^2\geq\left(\frac{j}{2}\right)^{1/2} \text{ for each }j\geq 1.
\end{equation}
So applying the argument leading to \eqref{eq:d_nW_ndensitynear0} to the product in \eqref{eq:F_Wdensity}, we obtain (recall $m=0$)
\begin{equation}\label{eq:f_Wsmallx2}
	f_W(x)=K\exp\left[-\kappa_1x^4-\kappa_2x^2\right]\exp\left\{\sum_{k=2}^{\infty}\frac{(-1)^{k-1}x^{2k}\sum_{j\geq1}a_j^{-2k}}{k}\right\},~|x|<\left(\frac{1}{2}\right)^{1/4}.
\end{equation}
Comparing \eqref{eq:f_Wsmallx1} and \eqref{eq:f_Wsmallx2}, we get
\begin{equation}\label{eq:kappaa_j}
	K=c, \kappa_2=0, \kappa_1+\frac{\sum_{j\geq 1}a_j^{-4}}{2}=1.
\end{equation}

Finally, we show that $d_{n_k}W_{n_k}$ converges in distribution to $W$ as $k\rightarrow\infty$. Suppose that there is another subsequence of $\{n_k:k\in\mathbb{N}\}$, $\{\tilde{m}_l:l\in\mathbb{N}\}$, and a random variable $\tilde{W}$ such that
\begin{equation}\label{eq:W_tildem_llimit}
	d_{\tilde{m}_l}W_{\tilde{m}_l}\Longrightarrow \tilde{W}:=C_9X \text{ as }l\rightarrow\infty \text{ for some } C_9\in(0,\infty),
\end{equation}
where we have used \eqref{eq:tildeW_n_klimit}.
Then by \eqref{eq:W_m_llimit} and \eqref{eq:W_tildem_llimit}, we have
\begin{equation}\label{eq:tildeWtoW}
	\tilde{W}\overset{d}{=}\rho W \text{ for some }\rho\in(0,\infty).
\end{equation}
The arguments leading to \eqref{eq:f_p_qconv}, \eqref{eq:alpha_jconv} and \eqref{eq:kappaa_j} also imply that there exists a subsequence of $\{\tilde{m}_l:l\in\mathbb{N}\}$, $\{\tilde{p}_q:q\in\mathbb{N}\}$, such that
\begin{equation}\label{eq:Wtildep_qlimit}
	\lim_{q\rightarrow\infty}f_{d_{\tilde{p}_q}W_{\tilde{p}_q}}(x)=f_{\tilde{W}}(x)=\tilde{K}\exp[-\tilde{\kappa}_1x^4]\prod_{j\geq 1}\left[\left(1+\frac{x^2}{\tilde{a}_j^2}\right)\exp\left(-\frac{x^2}{\tilde{a}_j^2}\right)\right] 
\end{equation}
uniformly on compact subsets of $\mathbb{C}$ where
\begin{align}
	&\lim_{q\rightarrow\infty}\alpha_{j,\tilde{p}_q}\left[-u_4(Y_{\tilde{p}_q})/4!\right]^{1/4}=\tilde{a}_j \text{ for each }j\geq 1,\\
	&\tilde{\kappa}_1+\frac{\sum_{j\geq 1}\tilde{a}_j^{-4}}{2}=1.\label{eq:tildekappaa_j}
\end{align}
The relation between $\tilde{W}$ and $W$ in \eqref{eq:tildeWtoW}, \eqref{eq:F_Wdensity} and \eqref{eq:kappaa_j} give
\begin{equation}\label{eq:tildeWdensity}
	f_{\tilde{W}}(x)=\frac{1}{\rho}f_W\left(\frac{x}{\rho}\right)=\frac{K}{\rho}\exp\left[-\kappa_1\frac{x^4}{\rho^4}\right]\prod_{j\geq 1}\left[\left(1+\frac{x^2}{\rho^2a_j^2}\right)\exp\left(-\frac{x^2}{\rho^2a_j^2}\right)\right],~\forall x\in\mathbb{C}.
\end{equation}	
From \eqref{eq:Wtildep_qlimit} and \eqref{eq:tildeWdensity}, we get
\begin{equation}
	\tilde{\kappa}_1=\kappa_1/\rho^4, \tilde{a}_j=\rho a_j \text{ for each }j\geq 1.
\end{equation}
Plugging this into \eqref{eq:tildekappaa_j} and using \eqref{eq:kappaa_j}, we get $\rho=1$. Therefore
\begin{align}
	&d_{n_k}W_{n_k}\Longrightarrow W \text{ as }k\rightarrow\infty,\\
	&\lim_{k\rightarrow\infty}\alpha_{j,n_k}\left[-u_4(Y_{n_k})/4!\right]^{1/4}=a_j \text{ for each }j\geq 1.
\end{align}
This completes the proof of Theorem \ref{thm1} by applying Lemma \ref{lem:3sequences}.
\end{proof}

\subsection{Proof of Corollary \ref{cor1}}\label{subsec:cor}
We conclude this section with a proof of Corollary \ref{cor1}.
\begin{proof}[Proof of Corollary \ref{cor1}]
Let $W$ be a random variable with density function
\begin{equation}
	f_W(x)=\exp(-x^4)/\int_{-\infty}^{\infty}\exp(-t^4)dt,~\forall x\in\mathbb{R}.
\end{equation}
Then we have
\begin{equation}
	\mathbb{E}(W^2)=\frac{\int_{-\infty}^{\infty}x^2\exp(-x^4)dx}{\int_{-\infty}^{\infty}\exp(-x^4)dx}=\frac{\Gamma(3/4)}{\Gamma(1/4)}.
\end{equation}

So to prove Corollary \ref{cor1}, by Theorem \ref{thm1}, it suffices to show that there exists $\beta_1(d)\in(0,\beta_c(d)]$ such that for any $\beta\in[0,\beta_1(d))$, we have that
	\begin{equation}\label{eq:alphatoinfty}
		\lim_{n\rightarrow\infty}\alpha_{1,n}\left[-u_4(Y_n)/4!\right]^{1/4}=\infty.
	\end{equation}
By \cite{Rue71} or Theorem A of \cite{PR20}, there exists $\beta_1(d)\in(0,\beta_c(d)]$ such that for any $\beta\in[0,\beta_1(d))$, there is $C_{10}=C_{10}(d,\beta)\in(0,\infty)$ satisfying
\begin{equation}\label{eq:alpha_1lowerbd}
	\alpha_{1,n}\geq C_{10} \text{ uniformly in }n.
\end{equation}
Note that 
\begin{equation}
	\langle \exp(zY_n)\rangle_{\Lambda_n,\beta}=0 \text{ is equivalent to }\sum_{\sigma\in\{-1,+1\}^{\Lambda_n}}\exp\left[\beta\sum_{\{u,v\}}\sigma_u\sigma_v+z\sum_{u\in\Lambda_n}\sigma_u\right]=0.
\end{equation}
We also note that
\begin{equation}
	\langle\exp(zY_n)\rangle_{\Lambda_n,\beta}=(-1)^{k|\Lambda_n|}\left\langle\exp\left((z+ik\pi)Y_n\right)\right\rangle_{\Lambda_n,\beta} \text{ for any }k\in\mathbb{Z}
\end{equation}
since $Y_n$ and $|\Lambda_n|$ always have the same parity. So the arguments from \eqref{eq:partf} to \eqref{eq:partffactor} in the Appendix imply that exactly $|\Lambda_n|$ roots of $\langle \exp(zY_n)\rangle_{\Lambda_n,\beta}$ lie in the interval $i[0,\pi]$:
\begin{equation}
	\alpha_{1,n}=\frac{\theta_{1,n}}{2}, \alpha_{2,n}=\frac{\theta_{2,n}}{2}, \dots,\alpha_{|\Lambda_n|,n}=\frac{\theta_{|\Lambda_n|,n}}{2} \text{ where }0<\theta_{1,n}\leq\theta_{2,n}\leq\dots\leq\theta_{|\Lambda_n|,n}<2\pi.
\end{equation}
Therefore, all roots of $\langle \exp(zY_n)\rangle_{\Lambda_n,\beta}$ are
\begin{equation}
	\left\{(\pm i)(\theta_{l,n}/2+k\pi): l\in\{1,2,\dots,|\Lambda_n|\}, k\in\{0,1,2,\dots\}\right\}.
\end{equation}
This implies that for any $\beta\geq 0$, using \eqref{eq:cumdef}, we have
\begin{equation}\label{eq:u_4lowerbd}
	-u_4(Y_n)/4!=\frac{1}{2}\sum_{j\geq 1}\alpha_{j,n}^{-4}=\frac{1}{2}\sum_{l=1}^{|\Lambda_n|}\sum_{k=0}^{\infty}(\frac{\theta_{l,n}}{2}+k\pi)^{-4}\geq\frac{|\Lambda_n|}{2}\sum_{k=0}^{\infty}(\pi+k\pi)^{-4}\geq C_{11}|\Lambda_n|,
\end{equation}
where $C_{11}\in(0,\infty)$. Clearly, \eqref{eq:alphatoinfty} follows from \eqref{eq:alpha_1lowerbd} and \eqref{eq:u_4lowerbd}.
\end{proof}

\appendix
\section{Limit distribution of Lee-Yang zeros when $\beta<\beta_c(d)$}

\renewcommand*{\thetheorem}{\Alph{theorem}}
\setcounter{theorem}{0}
In this appendix, we prove that when $\beta<\beta_c(d)$, the limiting distribution of Lee-Yang zeros has no mass in an arc containing $\exp(i0)$ of the unit circle. This is stated as a conjecture in Section 1.3 of \cite{CHJR19}. As can be seen from below, the proof follows essentially from Theorem 1.2 of \cite{Gre60} and Theorem 1.5 of \cite{Ott20}. We present a slightly different and relatively self-contained proof here which might be better suited to the context.  We came up with this proof before we knew of the existence of \cite{Gre60}.

Let $\Lambda\subseteq\mathbb{Z}^d$ be finite. The Ising model on $\Lambda$ at inverse temperature $\beta\geq0$ with free boundary conditions and external field $h\in {\mathbb R}$ is defined by the probability measure $\mathbb{P}_{\Lambda,\beta,h}$ on $\{-1,+1\}^{\Lambda}$ such that for each $\sigma\in\{-1,+1\}^{\Lambda}$
\begin{equation}\label{eq:Isingdef}
	\mathbb{P}_{\Lambda,\beta,h}(\sigma):=\frac{\exp\left[\beta\sum_{\{u,v\}}\sigma_u\sigma_v+h\sum_{u\in\Lambda}\sigma_u\right]}{Z_{\Lambda,\beta,h}},
\end{equation}
where the first sum is over all  nearest-neighbor edges in $\Lambda$, and
$Z_{\Lambda,\beta,h}$ is the partition function that makes \eqref{eq:Isingdef}
a probability measure.  In this appendix, we will consider complex $h\in\mathbb{C}$. A famous result due to Lee and Yang \cite{LY52} is that $Z_{\Lambda,\beta,h}\neq 0$ if $h\notin i\mathbb{R}$ where $i\mathbb{R}$ denotes the pure imaginary axis. So the fraction in \eqref{eq:Isingdef} is well-defined if $h\notin i\mathbb{R}$ but it could take a complex value, and thus $\mathbb{P}_{\Lambda,\beta,h}$ is a complex measure. Let $\langle \cdot\rangle_{\Lambda,\beta,h}$ denote the expectation with respect to $\mathbb{P}_{\Lambda,\beta,h}$. For example, the magnetization at $u\in\Lambda$ is defined by
\begin{equation}\label{eq:expectationdef}
	\langle\sigma_u\rangle_{\Lambda,\beta,h}:=\frac{\sum_{\sigma\in\{-1,+1\}^{\Lambda}}\sigma_u\exp\left[\beta\sum_{\{u,v\}}\sigma_u\sigma_v+h\sum_{u\in\Lambda}\sigma_u\right]}{Z_{\Lambda,\beta,h}}.
\end{equation}

Let $\Lambda_n:=[-n,n]^d\cap\mathbb{Z}^d$ and $E_n$ be the set of all nearest-neighbor edges $\{u,v\}$ with $u,v\in\Lambda_n$. Note that
\begin{align}
	Z_{\Lambda_n,\beta,h}&=\sum_{\sigma\in\{-1,+1\}^{\Lambda_n}}\exp\left[\beta\sum_{\{u,v\}\in E_n}\sigma_u\sigma_v+h\sum_{u\in\Lambda_n}\sigma_u\right]\label{eq:partf}\\
	&=\exp\left[\beta|E_n|+h|\Lambda_n|\right]\sum_{\sigma\in\{-1,+1\}^{\Lambda_n}}\exp\left[\beta\sum_{\{u,v\}\in E_n}(\sigma_u\sigma_v-1)+h\sum_{u\in\Lambda_n}(\sigma_u-1)\right].\label{eq:partition}
\end{align}
Let $z=e^{-2h}$ throughout the appendix and write $Z_{\Lambda_n,\beta}(z):=Z_{\Lambda_n,\beta,h}$. Then the outmost sum in \eqref{eq:partition} is a polynomial of $z$ with degree $|\Lambda_n|$. So by the fundamental theorem of algebra, $Z_{\Lambda_n,\beta}(z)$ has exactly $|\Lambda_n|$ complex roots. The Lee-Yang circle theorem~\cite{LY52} says that these $|\Lambda_n|$ roots are on the unit circle $\partial \mathbb{D}$ where $\mathbb{D}:=\{z\in\mathbb{C}:|z|<1\}$ is the unit disk. So we may assume that these roots are
\begin{equation}\label{eq:roots}
	\exp(i\theta_{1,n}),\exp(i\theta_{2,n}),\dots,\exp(i\theta_{|\Lambda_n|,n}) \text{ where } 0<\theta_{1,n}\leq\theta_{2,n}\leq\dots\leq\theta_{|\Lambda_n|,n}<2\pi.
\end{equation}
Note that we have used the fact that $Z_{\Lambda_n,\beta}(1)>0$. By the spin-flip symmetry, $Z_{\Lambda_n,\beta,-h}=Z_{\Lambda_n,\beta,h}$ for any $h\in\mathbb{C}$. As a result, those $|\Lambda_n|$ roots in \eqref{eq:roots} are symmetric with respect to the real-axis. Combining \eqref{eq:partition} and \eqref{eq:roots}, we have
\begin{equation}\label{eq:partffactor}
	Z_{\Lambda_n,\beta}(z)=\exp[\beta|E_n|-(\ln z)|\Lambda_n|/2]\prod_{j=1}^{|\Lambda_n|}\left[z-\exp(i\theta_{j,n})\right],~ z\in\mathbb{C}\setminus(-\infty,0].
\end{equation}
Here we take out $(-\infty,0]$ from $\mathbb{C}$ so that $h=-(\ln z)/2$ is analytic in this slit domain. Therefore,
\begin{equation}\label{eq:deri}
	\frac{\partial \ln Z_{\Lambda_n,\beta}(z)}{\partial z}=-\frac{|\Lambda_n|}{2z}+\sum_{j=1}^{|\Lambda_n|}\frac{1}{z-\exp(i\theta_{j,n})},~\forall z\in\mathbb{D}\setminus(-1,0].
\end{equation}
It is easy to see that
\begin{equation}\label{eq:deri1}
	\left\langle\sum_{u\in\Lambda_n}\sigma_u\right\rangle_{\Lambda_n,\beta,h}=\frac{\partial \ln Z_{\Lambda_n,\beta,h}}{\partial h}=\frac{\partial \ln Z_{\Lambda_n,\beta}(z)}{\partial z}\frac{\partial z}{\partial h}=-2z\frac{\partial \ln Z_{\Lambda_n,\beta}(z)}{\partial z},~\forall z\in\mathbb{D}\setminus(-1,0].
\end{equation}

We define the \textbf{average magnetization density} in $\Lambda_n$ by
\begin{equation}
	m_{\Lambda_n}(z):=\frac{\left\langle\sum_{u\in\Lambda_n}\sigma_u\right\rangle_{\Lambda_n,\beta,h}}{|\Lambda_n|},~\forall z\in\mathbb{C}\setminus\partial\mathbb{D},
\end{equation}
where we have dropped the $\beta$ dependence in $m_{\Lambda_n}(z)$.  By \eqref{eq:deri} and \eqref{eq:deri1}, we have
\begin{equation}\label{eq:amd}
	m_{\Lambda_n}(z)=1-\frac{2z}{|\Lambda_n|}\sum_{j=1}^{|\Lambda_n|}\frac{1}{z-\exp(i\theta_{j,n})}=\frac{1}{|\Lambda_n|}\sum_{j=1}^{|\Lambda_n|}\frac{\exp(i\theta_{j,n})+z}{\exp(i\theta_{j,n})-z},~\forall z\in\mathbb{D}\setminus(-1,0].
\end{equation}
By using the definition of $\langle \cdot\rangle_{\Lambda,\beta,h}$ (see \eqref{eq:expectationdef}), we know that $m_{\Lambda_n}(z)$ is a rational function of $z$ with poles on $\partial \mathbb{D}$, and thus \eqref{eq:amd} holds for each $z\in \mathbb{D}$. We define the empirical distribution
\begin{equation}
	\mu_n:=\frac{1}{|\Lambda_n|}\sum_{j=1}^{|\Lambda_n|}\delta_{\exp(i\theta_{j,n})},
\end{equation}
where $\delta_{\exp(i\theta_{j,n})}$ is the unit Dirac point measure at $\exp(i\theta_{j,n})$. Then 
\begin{equation}\label{eq:mint}
	m_{\Lambda_{n}}(z)=\int_{\partial\mathbb{D}}\frac{e^{i\theta}+z}{e^{i\theta}-z}d\mu_{n}(e^{i\theta}),~\forall z\in\mathbb{D}.
\end{equation}

Since $\mu_n$'s live on the unit circle, $\{\mu_n:n\in\mathbb{N}\}$ is tight. So there is a subsequence of $\mathbb{N}$, $\{n_k:k\in\mathbb{N}\}$, such that $\mu_{n_k}\Longrightarrow \mu$ as $k\rightarrow\infty$ where $\mu$ is some probability measure on $\partial \mathbb{D}$. We will see later that $\mu$ is actually unique. Therefore,
\begin{equation}\label{eq:sublimit}
	m_{\Lambda_{n_k}}(z)=\int_{\partial\mathbb{D}}\frac{e^{i\theta}+z}{e^{i\theta}-z}d\mu_{n_k}(e^{i\theta})\rightarrow\int_{\partial\mathbb{D}}\frac{e^{i\theta}+z}{e^{i\theta}-z}d\mu(e^{i\theta}) \text{ as }k\rightarrow\infty,~\forall z\in \mathbb{D}.
\end{equation}

Note that 
\begin{equation}
	\Re m_{\Lambda_n}(z)=\int_{\partial\mathbb{D}}\Re\frac{e^{i\theta}+z}{e^{i\theta}-z}d\mu_{n}(e^{i\theta})=\int_{\partial\mathbb{D}}\frac{1-|z|^2}{|e^{i\theta}-z|^2}d\mu_n(e^{i\theta})>0,~\forall z\in\mathbb{D},
\end{equation}
and $m_{\Lambda_n}(0)=1$. So $m_{\Lambda_n}$ is a \textbf{Herglotz function}. See Section 8.4 of \cite{RR94} for more details. From \eqref{eq:mint}, we know
\begin{equation}
	|m_{\Lambda_n}(z)|\leq \frac{1+r}{1-r}, \text{ for any } z \text{ satisfying } |z|\leq r<1.
\end{equation}
So $\{m_{\Lambda_n}:n\in\mathbb{N}\}$ is locally uniformly bounded. It is well-known that
\begin{equation}
	\lim_{n\rightarrow\infty}m_{\Lambda_n}(z)=\langle \sigma_0\rangle_{\mathbb{Z}^d,\beta,h},~\forall z\in(0,1),
\end{equation}
where $\langle \cdot\rangle_{\mathbb{Z}^d,\beta,h}$ is the expectation with respect to the unique infinite volume measure when $\beta\geq 0$ and $h>0$ (see, e.g., Proposition 3.29 and Theorem 3.46 of \cite{FV17}). So by Vitali's theorem (see, e.g., Theorem B.25 of \cite{FV17}),
\begin{equation}\label{eq:fulllimit}
	m(z):=\lim_{n\rightarrow\infty}m_{\Lambda_n}(z)
\end{equation}
exists locally uniformly on $\mathbb{D}$ and $m$ is a Herglotz function. Comparing \eqref{eq:sublimit} and \eqref{eq:fulllimit}, we obtain
\begin{equation}
	m(z)=\int_{\partial\mathbb{D}}\frac{e^{i\theta}+z}{e^{i\theta}-z}d\mu(e^{i\theta}),~\forall z\in\mathbb{D}.
\end{equation}
We define
\begin{equation}\label{eq:odd}
	m(z):=-m(1/z),~\forall z\in\mathbb{C}\setminus\bar{\mathbb{D}}.
\end{equation}
Since $\langle\sigma_x\rangle_{\Lambda_n,\beta,h}=-\langle\sigma_x\rangle_{\Lambda_n,\beta,-h}$ for any $x\in\Lambda_n$ and $h\in\mathbb{C}\setminus i\mathbb{R}$, we get
\begin{equation}\label{eq:magoutD}
	m(z)=\lim_{n\rightarrow\infty}m_{\Lambda_n}(z),~\forall z\in \mathbb{C}\setminus\bar{\mathbb{D}}.
\end{equation}
The following Stieltjes inversion formula on page 12 of \cite {RR94} will be very import to our analysis of $\mu$.
\begin{theorem}[Stieltjes Inversion Formula]\label{thm:Sti}
	Let $\gamma:=\{e^{it}: a<t<b\}$ be an open arc on $\partial \mathbb{D}$ with endpoints $e^{ia}$ and $e^{ib}$, $0<b-a<2\pi$. Then
	\begin{equation}
		\mu(\gamma)+\frac{1}{2}\mu(\{e^{ia}\})+\frac{1}{2}\mu(\{e^{ib}\})=\lim_{r\uparrow1}\frac{1}{2\pi}\int_a^b \Re m(re^{i\theta}) d\theta.
	\end{equation}
\end{theorem}

In particular, Theorem \ref{thm:Sti} implies the $\mu$ that we obtained from the subsequential limit is unique, that is, we have $\mu_n\Longrightarrow \mu$ as $n\rightarrow\infty$. We call $\mu$ the \textbf{limiting distribution of Lee-Yang zeros}. Note that $\mu$ is actually a function of $\beta$. Let $\beta_c(d)$ be the critical inverse temperature of the Ising model on $\mathbb{Z}^d$. We are ready to prove the main result about $\mu$.
\begin{theorem}
	For any $d\geq 1$ and any $\beta\in[0,\beta_c(d))$, there is $\epsilon>0$ (which only depends on $\beta$ and $d$) such that
	\begin{equation}
		\mu\left(\{e^{it}:-\epsilon<t<\epsilon\}\right)=0.
	\end{equation}
\end{theorem}
\begin{proof}
	By Theorem \ref{thm:Sti},
	\begin{align}
		&\mu\left(\{e^{it}:-\epsilon<t<\epsilon\}\right)+\frac{1}{2}\mu\left(\{e^{i\epsilon}\}\right)+\frac{1}{2}\mu\left(\{e^{i\epsilon}\}\right)=\lim_{r\uparrow1}\frac{1}{2\pi}\int_{-\epsilon}^{\epsilon} \Re m(re^{i\theta}) d\theta\\
		&\qquad=\lim_{r\uparrow1}\frac{1}{2\pi}\left[\int_{-\epsilon}^{0} \Re m(re^{i\theta}) d\theta+\int_{0}^{\epsilon} \Re m(re^{i\theta}) d\theta\right]\\
		&\qquad=\lim_{r\uparrow1}\frac{1}{2\pi}\left[\int_{0}^{\epsilon} \Re m(re^{-i\theta}) d\theta+\int_{0}^{\epsilon} \Re m(re^{i\theta}) d\theta\right].
	\end{align}
	By Theorem 1.5 of \cite{Ott20}, we have that $m(z)$ is complex analytic in a neighborhood of $z=1$. So we may pick $\epsilon$ small such that $m$ is analytic in $D(1,2\epsilon):=\{z\in\mathbb{C}:|z-1|<2\epsilon\}$. Then both $\Re m(re^{-i\theta}) $ and $\Re m(re^{i\theta}) $ are bounded if $re^{-i\theta}$ and $re^{i\theta}$ are in $D(1,\epsilon)$. The dominated converge theorem implies that
	\begin{align}
		\mu\left(\{e^{it}:-\epsilon<t<\epsilon\}\right)+\frac{1}{2}\mu\left(\{e^{i\epsilon}\}\right)+\frac{1}{2}\mu\left(\{e^{i\epsilon}\}\right)\\
		=\frac{1}{2\pi}\int_{0}^{\epsilon}\left[ \lim_{r\uparrow 1}\Re m(re^{-i\theta}) +\lim_{r\uparrow 1} \Re m(re^{i\theta})\right] d\theta.\label{eq:mulimit}
	\end{align}
By \eqref{eq:odd} and \eqref{eq:magoutD}, and continuity of $m$ in $D(1,2\epsilon)$, we have
\begin{equation}
	\lim_{r\uparrow1}\Re m(re^{-i\theta})=-	\lim_{r\uparrow1}\Re m(r^{-1}e^{i\theta})=-\lim_{r\uparrow1}\Re m(re^{i\theta}).
\end{equation}
Plugging this into \eqref{eq:mulimit}, we get
\begin{equation}
	\mu\left(\{e^{it}:-\epsilon<t<\epsilon\}\right)+\frac{1}{2}\mu\left(\{e^{i\epsilon}\}\right)+\frac{1}{2}\mu\left(\{e^{i\epsilon}\}\right)=0,
\end{equation}
which concludes the proof of the theorem.
\end{proof}

\section*{Acknowledgements}
The research of the second author was partially supported by NSFC grant 11901394 and that of the third author by US-NSF grant DMS-1507019. The authors thank two anonymous reviewers for useful comments and suggestions.

\bibliographystyle{abbrv}
\bibliography{reference}

\end{document}